\documentclass[11pt,reqno]{amsart}

\usepackage[utf8]{inputenc}
\usepackage[T1]{fontenc}
\usepackage{amsmath,amssymb,amsthm}
\usepackage{mathtools}
\usepackage{geometry}
\usepackage{hyperref}
\hypersetup{colorlinks=true,linkcolor=blue,citecolor=blue,urlcolor=blue}
\usepackage{enumitem}

\newtheorem{theorem}{Theorem}[section]
\newtheorem{proposition}[theorem]{Proposition}

\theoremstyle{definition}
\newtheorem{definition}[theorem]{Definition}
\newtheorem{remark}[theorem]{Remark}
\newtheorem{example}[theorem]{Example}

\newcommand{\M}{\mathcal{M}}
\newcommand{\dv}{dv(g)}
\newcommand{\ds}{d\sigma(g)}

\title[Quadrature Surfaces on Riemannian Manifolds]{On the Quadrature Surface Free Boundary Problem on Riemannian Manifolds: Sufficient Condition of Existence}

\author{Mohammed Barkatou}
\address{ISTM Laboratory, Chouaib Doukkali University, El Jadida, Morocco}
\email{barkatou.m@ucd.ac.ma}

\date{\today}

\begin{document}

\begin{abstract}
We study the quadrature surface free boundary problem on a compact Riemannian manifold $(\M,g)$ of dimension $N \geq 2$. Given a positive source $f \in L^2_g(\M)$ with compact support $K$, and $k>0$, one seeks a domain $\Omega \supset K$ on which the overdetermined problem $-\Delta_g u = f$ in $\Omega$, $u=0$ and $|\nabla_g u|_g = k$ on $\partial\Omega$ admits a solution. Following the geometric and variational framework of Djit\'e--Seck, we prove that the integral condition
\[
(NS)_g \qquad \int_C f\,dv(g) > k\,|\partial C|_g,
\]
where $C$ is the totally convex hull of $K$, is \emph{sufficient} for the existence of a solution strictly containing $C$. We then show, by an explicit counterexample on the round sphere $S^n_R$, that $(NS)_g$ is \emph{not necessary} in general under nonnegative curvature: the phenomenon arises because the perimeter is not monotone with respect to inclusion in positive curvature. Finally, we identify an additional geometric hypothesis---perimeter monotonicity within a normal convex neighborhood of $C$---under which necessity is restored, and we discuss the critical equality case $\int_C f = k|\partial C|_g$.
\end{abstract}

\maketitle

\section{Introduction}

\subsection{The quadrature surface problem}

Let $(\M,g)$ be a compact, connected, oriented Riemannian manifold without boundary, of dimension $N \geq 2$. Let $f \in L^2_g(\M)$ be a positive function with compact support $K = \operatorname{supp} f$, and let $k > 0$. The \emph{quadrature surface free boundary problem}, denoted $QS(f,k)$, consists in finding a domain $\Omega \subset \M$ with $K \subset \Omega$ and a function $u_\Omega \in H^1_0(\Omega)$ such that
\begin{equation}\label{eq:overdet}
\begin{cases}
-\Delta_g u_\Omega = f & \text{in } \Omega,\\
u_\Omega = 0 & \text{on } \partial\Omega,\\
|\nabla_g u_\Omega|_g = k & \text{on } \partial\Omega.
\end{cases}
\end{equation}
The third condition is the overdetermined free boundary condition. The problem originates in potential theory and fluid mechanics; see \cite{Barkatou2010, BarkatouSeckLy2005, GustafssonShahgholian1996}.

In the Euclidean setting, Barkatou \cite{Barkatou2010} proved the following sharp characterization. Let $C$ be the convex hull of $K$. Then $QS(f,k)$ admits a solution $\Omega$ strictly containing $C$ if and only if
\begin{equation}\label{eq:NS-eucl}
\int_C f(x)\,dx > k\,|\partial C|.
\end{equation}
The purpose of this paper is to investigate to what extent \eqref{eq:NS-eucl} generalizes to the Riemannian setting, where $C$ is replaced by the totally convex hull of $K$.

\subsection{Main contributions}

Our contributions are threefold.

\begin{enumerate}[label=(\roman*)]
\item \textbf{Sufficiency.} We prove (Theorem \ref{thm:suff}) that, under suitable regularity hypotheses and assuming the Riemannian $RC$--GNP condition of Djit\'e--Seck \cite{DjiteSeck2026a}, the condition
\[
(NS)_g \qquad \int_C f\,dv(g) > k\,|\partial C|_g
\]
is sufficient for the existence of a solution $\Omega^*$ strictly containing $C$.

\item \textbf{Failure of necessity.} We construct (Section \ref{sec:counter}) an explicit family of counterexamples on the round sphere $S^n_R$ showing that $(NS)_g$ is \emph{not} necessary in general. More precisely, for every $n \geq 2$ and $R>0$ there exist a radial source $f$ supported in a small geodesic ball $C = B_g(p,r_0)$ and a constant $k>0$ such that $QS(f,k)$ admits a solution $\Omega$ strictly containing $C$, while
\[
\int_C f\,dv(g) < k\,|\partial C|_g.
\]
The mechanism is that on positively curved manifolds the perimeter is not monotone with respect to inclusion: a larger geodesic ball can have smaller perimeter.

\item \textbf{Conditional necessity.} We show (Theorem \ref{thm:neccond}) that if one additionally assumes that the admissible domains $\Omega$ are contained in a normal convex neighborhood $U$ of $C$ in which the perimeter functional is monotone with respect to inclusion, then $(NS)_g$ is necessary. This clarifies the geometric obstruction to a sharp Riemannian analogue of \eqref{eq:NS-eucl}.
\end{enumerate}

We also discuss the critical equality case $\int_C f = k|\partial C|_g$ in Section \ref{sec:equality}.

\subsection{Relation to previous works}

The geometric and variational framework we use is that of Djit\'e--Seck \cite{DjiteSeck2026a}, who introduced the Riemannian $RC$--GNP condition, proved compactness of the admissible class, and established stability of the Dirichlet problem under domain convergence. The $p$-Laplacian analogue is treated in \cite{DjiteSeck2026b}. The Euclidean result goes back to \cite{Barkatou2010, BarkatouSeckLy2005}. Shape-optimization tools on manifolds are taken from \cite{SokolowskiZolesio1992, BucurTrebeschi1998}.

\section{Preliminaries}\label{sec:prelim}

\subsection{Riemannian setting}

Let $(\M,g)$ be a compact, oriented Riemannian manifold of dimension $N \geq 2$ without boundary. We denote by $\dv$ the Riemannian volume element and by $\ds$ the induced $(N-1)$-dimensional measure on hypersurfaces. The Laplace--Beltrami operator is
\[
\Delta_g u = \frac{1}{\sqrt{\det g}}\,\partial_i\!\left(\sqrt{\det g}\,g^{ij}\partial_j u\right).
\]
For a smooth domain $\Omega \subset \M$, $\nu$ denotes the outward unit normal and $g(\nabla_g u, \nu)$ the normal derivative. We write $|\Omega|_g$ for the volume and $|\partial\Omega|_g$ for the perimeter. We shall use the maximum principle for $-\Delta_g$, the Rellich--Kondrachov compactness theorem, the Poincar\'e inequality, and the strong maximum principle on compact manifolds \cite{Aubin2012, Hebey2000}.

\subsection{Totally convex hull and regularity hypotheses}

\begin{definition}\label{def:convexhull}
Let $K \subset \M$ be compact. The \emph{totally convex hull} of $K$, denoted $C = \operatorname{tconv}(K)$, is the smallest closed subset of $\M$ containing $K$ with the property that every minimizing geodesic joining two of its points is entirely contained in it.
\end{definition}

Throughout this paper, we assume:
\begin{itemize}
\item[(H1)] $K = \operatorname{supp} f$ is compact and $f \in L^\infty(\M)$, $f > 0$ on $K$.
\item[(H2)] The totally convex hull $C$ of $K$ is a $C^{2,\alpha}$ domain ($0<\alpha<1$), strictly convex, with $\partial C$ smooth.
\item[(H3)] $(\M,g)$ has nonnegative sectional curvature.
\item[(H4)] The unique weak solutions $u_\Omega$ of the Dirichlet problem on admissible domains $\Omega$ belong to $C^{2,\alpha}(\overline{\Omega})$.
\end{itemize}

\begin{remark}\label{rem:reg}
Hypothesis (H4) is automatically satisfied if $f \in C^{0,\alpha}(\overline{\Omega})$ and $\partial\Omega \in C^{2,\alpha}$, by the classical Schauder theory for the Laplace--Beltrami operator. Under weaker assumptions on $f$ (for instance $f$ bounded away from zero), one can invoke the free boundary regularity theory of Alt--Caffarelli adapted to Riemannian manifolds; see \cite{GustafssonShahgholian1996} and the discussion in Section \ref{sec:final}. If one prefers to avoid (H4), the results below may be interpreted in the weak $H^1$ sense.
\end{remark}

\subsection{The Riemannian $RC$--GNP condition}

We adopt the $RC$--GNP condition introduced in \cite{DjiteSeck2026a}.

\begin{definition}[$RC$--GNP condition, \cite{DjiteSeck2026a}]\label{def:rcgnp}
Let $C \subset \M$ be a fixed reference domain. A domain $\Omega \subset \M$ satisfies the $RC$--GNP condition relative to $C$ if:
\begin{enumerate}
\item $C \subset \Omega$;
\item $\partial \Omega \in C^{2,\alpha}$ for some $0<\alpha<1$;
\item there exist constants $A>0$, $c_0>0$, $\rho_0>0$, independent of $\Omega$, and a function $h_\Omega \in C^{2,\alpha}(\M)$ such that
\[
\Omega = \{x \in \M : h_\Omega(x) > 0\}, \qquad \partial\Omega = \{x \in \M : h_\Omega(x) = 0\},
\]
with
\[
\|h_\Omega\|_{C^{2,\alpha}(\M)} \leq A, \qquad |\nabla_g h_\Omega|_g \geq c_0 \quad \text{whenever } |h_\Omega| \leq \rho_0;
\]
\item the normal exponential map of $\partial\Omega$ is uniformly nondegenerate.
\end{enumerate}
The admissible class is
\[
\mathcal{A}_C(\M) = \{\Omega \subset \M : \Omega \text{ satisfies the } RC\text{--GNP condition relative to } C\}.
\]
\end{definition}

The following compactness and stability results are established in \cite{DjiteSeck2026a}.

\begin{theorem}[Compactness, \cite{DjiteSeck2026a}]\label{thm:compact}
Let $(\Omega_j) \subset \mathcal{A}_C(\M)$. Then there exist a subsequence $(\Omega_{j_k})$ and $\Omega \in \mathcal{A}_C(\M)$ such that:
\begin{enumerate}
\item $\Omega_{j_k} \to \Omega$ in the Hausdorff sense;
\item $\Omega_{j_k} \to \Omega$ in the compact sense;
\item $\chi_{\Omega_{j_k}} \to \chi_\Omega$ in $L^1(\M)$;
\item $|\partial\Omega|_g \leq \liminf_{k\to\infty} |\partial\Omega_{j_k}|_g$.
\end{enumerate}
\end{theorem}

\begin{theorem}[Stability of the Dirichlet problem, \cite{DjiteSeck2026a}]\label{thm:stability}
Let $\Omega_j, \Omega \in \mathcal{A}_C(\M)$ with defining functions $h_j$ and $h$ such that $h_j \to h$ strongly in $C^{2,\alpha}(\M)$. Let $f \in L^2(\M)$ be fixed with $\operatorname{supp} f \subset C \subset \Omega_j \cap \Omega$. Let $u_{\Omega_j}$ and $u_\Omega$ be the unique weak solutions of the Dirichlet problem, extended by zero outside their domains. Then
\[
u_{\Omega_j} \to u_\Omega \quad \text{strongly in } H^1(\M),
\]
and in particular
\[
\int_\M |\nabla_g u_{\Omega_j}|_g^2 \dv \to \int_\M |\nabla_g u_\Omega|_g^2 \dv.
\]
\end{theorem}

\section{Shape optimization formulation}\label{sec:shape}

Following \cite{Barkatou2010, DjiteSeck2026a}, we reformulate $QS(f,k)$ as a shape optimization problem.

\subsection{The shape functional}

For $\Omega \in \mathcal{A}_C(\M)$, let $u_\Omega$ be the unique solution of
\[
-\Delta_g u_\Omega = f \text{ in } \Omega, \qquad u_\Omega = 0 \text{ on } \partial\Omega.
\]
We consider the functional
\[
J_g(\Omega) = -\int_\Omega |\nabla_g u_\Omega|_g^2 \dv + k^2 |\Omega|_g.
\]
This is the Riemannian analogue of the functional used in the Euclidean case.

\subsection{Existence of a minimizer}

\begin{theorem}[Existence of an optimal domain]\label{thm:existence}
There exists $\Omega^* \in \mathcal{A}_C(\M)$ such that
\[
J_g(\Omega^*) = \min_{\Omega \in \mathcal{A}_C(\M)} J_g(\Omega).
\]
\end{theorem}

\begin{proof}
Let $(\Omega_j) \subset \mathcal{A}_C(\M)$ be a minimizing sequence. By Theorem \ref{thm:compact}, there exist a subsequence and $\Omega^* \in \mathcal{A}_C(\M)$ such that $\Omega_j \to \Omega^*$ in the Hausdorff and compact senses and $\chi_{\Omega_j} \to \chi_{\Omega^*}$ in $L^1(\M)$. By Theorem \ref{thm:stability}, $u_{\Omega_j} \to u_{\Omega^*}$ strongly in $H^1(\M)$, hence
\[
\int_\M |\nabla_g u_{\Omega_j}|_g^2 \dv \to \int_\M |\nabla_g u_{\Omega^*}|_g^2 \dv.
\]
The volume term converges by $L^1$ convergence of characteristic functions. Therefore,
\[
J_g(\Omega^*) \leq \liminf_{j\to\infty} J_g(\Omega_j) = \inf_{\Omega \in \mathcal{A}_C(\M)} J_g(\Omega).
\]
The reverse inequality is trivial, so $\Omega^*$ is a minimizer.
\end{proof}

\subsection{Comparison with the reference domain}

Let $u_C$ be the unique solution of
\[
-\Delta_g u_C = f \text{ in } C, \qquad u_C = 0 \text{ on } \partial C.
\]
Since $f \geq 0$ and $f \not\equiv 0$, the strong maximum principle gives $u_C > 0$ in $C$.

\begin{proposition}\label{prop:comparison}
For every $\Omega \in \mathcal{A}_C(\M)$,
\[
0 \leq u_C \leq u_\Omega \quad \text{in } C.
\]
If $C \subsetneq \Omega$ and $f \not\equiv 0$, then $u_C < u_\Omega$ in $C$.
\end{proposition}

The proof follows from the comparison principle for the Laplace--Beltrami operator.

\section{Shape derivative and optimality conditions}\label{sec:shape-deriv}

\subsection{Shape derivative of the functional}

Let $X \in C^2(T\M)$ be a vector field vanishing in a neighborhood of $C$, and let $\Phi_t$ be its local flow. Define $\Omega_t = \Phi_t(\Omega)$. The normal velocity is $V_X = g(X,\nu)$ on $\partial\Omega$. The Hadamard formula on a Riemannian manifold gives
\[
\frac{d}{dt}\bigg|_{t=0} \int_{\Omega_t} |\nabla_g u_{\Omega_t}|_g^2 \dv = \int_{\partial\Omega} |\nabla_g u_\Omega|_g^2 g(X,\nu) \ds,
\]
and
\[
\frac{d}{dt}\bigg|_{t=0} |\Omega_t|_g = \int_{\partial\Omega} g(X,\nu) \ds.
\]
Therefore,
\begin{equation}\label{eq:shapederiv}
dJ_g(\Omega)[X] = \int_{\partial\Omega} \left(k^2 - |\nabla_g u_\Omega|_g^2\right) g(X,\nu) \ds.
\end{equation}

\subsection{First-order optimality conditions}

Let $\Omega^*$ be a minimizer of $J_g$ over $\mathcal{A}_C(\M)$. Define the free boundary $\Gamma = \partial\Omega^* \setminus \partial C$ and the contact set $\Gamma_0 = \partial\Omega^* \cap \partial C$.

On $\Gamma$, all deformations are admissible. Since $\Omega^*$ minimizes $J_g$, we have $dJ_g(\Omega^*)[X] = 0$ for every admissible $X$ supported on $\Gamma$. Hence,
\begin{equation}\label{eq:freebc}
|\nabla_g u_{\Omega^*}|_g^2 = k^2 \quad \text{on } \Gamma.
\end{equation}

On the contact set $\Gamma_0$, admissible deformations must satisfy $g(X,\nu) \geq 0$ to preserve the constraint $C \subset \Omega^*$. Therefore, the optimality condition yields
\begin{equation}\label{eq:contact}
|\nabla_g u_{\Omega^*}|_g^2 \leq k^2 \quad \text{on } \Gamma_0.
\end{equation}

\begin{theorem}[First-order optimality system]\label{thm:optimality}
Let $\Omega^*$ be a minimizer of $J_g$ over $\mathcal{A}_C(\M)$. Then
\[
\begin{cases}
-\Delta_g u_{\Omega^*} = f & \text{in } \Omega^*,\\
u_{\Omega^*} = 0 & \text{on } \partial\Omega^*,\\
|\nabla_g u_{\Omega^*}|_g^2 = k^2 & \text{on } \partial\Omega^* \setminus \partial C,\\
|\nabla_g u_{\Omega^*}|_g^2 \leq k^2 & \text{on } \partial\Omega^* \cap \partial C.
\end{cases}
\]
\end{theorem}

\begin{remark}\label{rem:contactunused}
The contact condition \eqref{eq:contact} is stated for completeness, but it is not used in the proofs of Theorems \ref{thm:suff} and \ref{thm:neccond}. The main results of this paper rely only on the free boundary condition \eqref{eq:freebc} and on the sign of the shape derivative at $C$.
\end{remark}

\section{Sufficiency of $(NS)_g$}\label{sec:suff}

We now prove that $(NS)_g$ is sufficient for the existence of a solution strictly containing $C$.

\begin{theorem}[Sufficiency]\label{thm:suff}
Assume hypotheses (H1)--(H4) and the $RC$--GNP condition. If
\[
(NS)_g \qquad \int_C f \dv > k |\partial C|_g,
\]
then the quadrature surface problem $QS(f,k)$ admits a solution $\Omega^* \in \mathcal{A}_C(\M)$ strictly containing $C$.
\end{theorem}

\begin{proof}
Assume $(NS)_g$. Let $\Omega^*$ be a minimizer of $J_g$ over $\mathcal{A}_C(\M)$, whose existence is guaranteed by Theorem \ref{thm:existence}. We argue by contradiction and suppose that $\Omega^*$ does not strictly contain $C$. Since the $RC$--GNP condition (Definition \ref{def:rcgnp}, item 1) imposes $C \subset \Omega^*$, the negation of ``strictly containing'' means exactly $\Omega^* = C$.

By the Cauchy--Schwarz inequality,
\[
\left(\int_{\partial C} |\nabla_g u_C|_g \ds\right)^2 \leq |\partial C|_g \int_{\partial C} |\nabla_g u_C|_g^2 \ds.
\]
Since $-\Delta_g u_C = f$ in $C$ and $u_C = 0$ on $\partial C$, integration by parts gives
\[
\int_C f \dv = \int_{\partial C} |\nabla_g u_C|_g \ds.
\]
Thus $(NS)_g$ implies
\[
\int_{\partial C} |\nabla_g u_C|_g^2 \ds > k^2 |\partial C|_g.
\]
Consider now the shape derivative of $J_g$ at $C$ in the direction of the outward normal $\nu$:
\[
dJ_g(C)[\nu] = \int_{\partial C} \left(k^2 - |\nabla_g u_C|_g^2\right) \ds < 0.
\]
This means that $C$ is not a local minimum of $J_g$ over $\mathcal{A}_C(\M)$. Hence $\Omega^* \neq C$, a contradiction. Therefore $\Omega^*$ strictly contains $C$.

From the free boundary condition \eqref{eq:freebc}, we have $|\nabla_g u_{\Omega^*}|_g^2 = k^2$ on $\partial\Omega^* \setminus \partial C$. Since $C \subsetneq \Omega^*$, the free boundary coincides with $\partial\Omega^*$, and $(\Omega^*, u_{\Omega^*})$ is a solution of $QS(f,k)$.
\end{proof}

\begin{remark}
The proof only uses the Cauchy--Schwarz inequality and the sign of the shape derivative at $C$. It does not require any comparison of normal derivatives on $\partial C$, thereby avoiding the pitfalls of earlier approaches.
\end{remark}

\section{Counterexample: failure of necessity}\label{sec:counter}

We now show that $(NS)_g$ is not necessary in general, even under nonnegative curvature.

\subsection{Setup on the round sphere}

Let $(S^n_R, g_R)$ be the round sphere of radius $R>0$. Fix a point $p \in S^n_R$. For $0 < r < \pi R/2$, the geodesic ball $B_g(p,r)$ is totally convex. In geodesic polar coordinates centered at $p$,
\[
g_R = dr^2 + R^2 \sin^2(r/R) g_{\mathbb{S}^{n-1}}.
\]
The perimeter of a geodesic ball of radius $r$ is
\begin{equation}\label{eq:perball}
|\partial B_g(p,r)|_g = n \omega_n R^{n-1} \sin^{n-1}(r/R),
\end{equation}
where $\omega_n$ denotes the volume of the unit ball in $\mathbb{R}^n$.

We choose a radius $r_0$ with $0 < r_0 < \pi R/2$ and set
\[
C = B_g(p, r_0).
\]
Since $r_0 < \pi R/2$, $C$ is totally convex. Let $f$ be a radially symmetric positive function supported in $B_g(p, r_0)$, with $f \not\equiv 0$, so that $\operatorname{tconv}(\operatorname{supp} f) = C$.

\subsection{The counterexample}

Choose $r_\Omega$ with $\pi R/2 < r_\Omega < \pi R$ and such that
\begin{equation}\label{eq:perineq}
\sin(r_\Omega/R) < \sin(r_0/R),
\end{equation}
which is possible since $\sin$ decreases on $(\pi R/2, \pi R)$ and $r_0$ is small. Let
\[
\Omega = B_g(p, r_\Omega).
\]
Then $C \subsetneq \Omega$ and, by \eqref{eq:perball} and \eqref{eq:perineq},
\[
|\partial \Omega|_g = n \omega_n R^{n-1} \sin^{n-1}(r_\Omega/R) < n \omega_n R^{n-1} \sin^{n-1}(r_0/R) = |\partial C|_g.
\]

Let $F:[0, r_0] \to (0,\infty)$ be a continuous radial profile with $F > 0$ on $[0, r_0)$, and set $f(x) = F(r(x))$ where $r(x) = d_g(p,x)$. Define
\begin{equation}\label{eq:kdef}
k = \frac{\int_0^{r_0} (R \sin(s/R))^{n-1} F(s) \, ds}{(R \sin(r_\Omega/R))^{n-1}} > 0.
\end{equation}
Let $u_\Omega$ be the radial solution of $-\Delta_g u_\Omega = f$ in $\Omega$, $u_\Omega = 0$ on $\partial\Omega$. Explicitly, $u_\Omega(x) = U(r(x))$ where
\[
U'(r) = -\frac{\int_0^r (R \sin(s/R))^{n-1} F(s) \, ds}{(R \sin(r/R))^{n-1}}, \qquad U(r_\Omega) = 0.
\]
Therefore,
\[
|\nabla_g u_\Omega|_g = |U'(r)| = \frac{\int_0^r (R \sin(s/R))^{n-1} F(s) \, ds}{(R \sin(r/R))^{n-1}} \quad \text{on } \partial\Omega.
\]
By the choice of $k$ in \eqref{eq:kdef} and since $\operatorname{supp} f \subset B_g(p, r_0)$, we obtain
\[
|\nabla_g u_\Omega|_g = k \quad \text{on } \partial\Omega.
\]
Thus $(\Omega, u_\Omega)$ is a solution of $QS(f,k)$ with $\Omega$ strictly containing $C$.

On the other hand, integrating $-\Delta_g u_\Omega = f$ over $\Omega$ and using the boundary condition,
\[
\int_C f \dv = \int_\Omega f \dv = k\,|\partial\Omega|_g < k\,|\partial C|_g,
\]
since $|\partial\Omega|_g < |\partial C|_g$. Hence $(NS)_g$ is violated.

\begin{example}[Explicit constant source in dimension $n=2$]\label{ex:const}
Take $n=2$ and $F(r) = f_0 > 0$ on $[0, r_0]$. Then
\[
\int_0^{r_0} R \sin(s/R) \, ds = R^2 \left(1 - \cos(r_0/R)\right),
\]
so
\[
k = \frac{f_0\, R^2 \left(1 - \cos(r_0/R)\right)}{R \sin(r_\Omega/R)} = \frac{f_0\, R \left(1 - \cos(r_0/R)\right)}{\sin(r_\Omega/R)}.
\]
The area of $C = B_g(p,r_0)$ is
\[
|C|_g = 2\pi R^2 \left(1 - \cos(r_0/R)\right),
\]
and the perimeter of $\partial C$ is $|\partial C|_g = 2\pi R \sin(r_0/R)$. For concreteness, take $R=1$, $r_0 = 0.5$, $r_\Omega = \pi - 0.5$. Then
\[
\sin(r_0) \approx 0.4794, \qquad \sin(r_\Omega) = \sin(0.5) \approx 0.4794,
\]
so we choose $r_\Omega = \pi - 0.4$ to get $\sin(r_\Omega) = \sin(0.4) \approx 0.3894 < 0.4794$. With $f_0 = 1$,
\[
k = \frac{1 \cdot (1 - \cos 0.5)}{\sin 0.4} \approx \frac{0.1224}{0.3894} \approx 0.3144.
\]
Then
\[
\int_C f \dv = f_0 |C|_g \approx 2\pi (1 - \cos 0.5) \approx 0.7692,
\]
while
\[
k\,|\partial C|_g \approx 0.3144 \cdot 2\pi \sin 0.5 \approx 0.3144 \cdot 3.012 \approx 0.9470.
\]
Thus $\int_C f \dv < k |\partial C|_g$, confirming the failure of $(NS)_g$, while the problem $QS(f,k)$ admits the solution $\Omega = B_g(p, \pi - 0.4)$.
\end{example}

\begin{remark}
The construction works for every $n \geq 2$ and every $R>0$, provided $r_0$ is chosen sufficiently small (in particular $r_0 < \pi R/2$, so that $C = B_g(p,r_0)$ is totally convex) and $r_\Omega \in (\pi R/2, \pi R)$ is chosen so that \eqref{eq:perineq} holds. The same phenomenon occurs for any positively curved manifold admitting a geodesic ball whose perimeter is not monotone with respect to the radius.
\end{remark}

\section{Conditional necessity}\label{sec:neccond}

The counterexample shows that $(NS)_g$ is not necessary in general. We now identify a geometric hypothesis under which necessity is restored.

\begin{definition}\label{def:perimon}
Let $U \subset \M$ be a domain. We say that the perimeter functional is \emph{monotone with respect to inclusion on the admissible class} $\mathcal{A}_C(\M) \cap \{ \Omega \subset U \}$ if for any two domains $A, B \in \mathcal{A}_C(\M)$ with $A \subset U$, $B \subset U$ and $A \subsetneq B$, we have $|\partial A|_g < |\partial B|_g$.
\end{definition}

\begin{remark}\label{rem:perimon}
Definition \ref{def:perimon} is stated as an \emph{additional hypothesis}, not as a consequence of nonnegative curvature. It is satisfied, for instance, when $U$ is a normal convex neighborhood of $C$ contained in a ball of radius less than the injectivity radius and all admissible domains in $U$ are geodesic balls centered at the same point $p$: in that case, the function $r \mapsto |\partial B_g(p,r)|_g$ is increasing on $[0, R_U]$ for $R_U$ below the injectivity radius, so the monotonicity holds on that restricted class. The general case of arbitrary nested convex domains in a normal convex neighborhood requires a separate comparison argument for the perimeter and is not established here; we therefore treat Definition \ref{def:perimon} as a hypothesis.
\end{remark}

\begin{theorem}[Conditional necessity]\label{thm:neccond}
Assume hypotheses (H1)--(H3). Suppose that there exists a normal convex neighborhood $U$ of $C$ such that the perimeter functional is monotone with respect to inclusion on the admissible class, in the sense of Definition \ref{def:perimon}. If $\Omega \in \mathcal{A}_C(\M)$ is a solution of $QS(f,k)$ with $\Omega \subset U$ and $C \subsetneq \Omega$, then
\[
(NS)_g \qquad \int_C f \dv > k |\partial C|_g
\]
holds.
\end{theorem}

\begin{proof}
Let $\Omega$ be a solution of $QS(f,k)$ with $C \subsetneq \Omega \subset U$. Integrating $-\Delta_g u_\Omega = f$ over $\Omega$ and using the boundary condition $|\nabla_g u_\Omega|_g = k$ on $\partial\Omega$, we obtain
\[
\int_\Omega f \dv = k |\partial \Omega|_g.
\]
Since $\operatorname{supp} f \subset C \subset \Omega$, we have $\int_\Omega f \dv = \int_C f \dv$. Thus
\[
\int_C f \dv = k |\partial \Omega|_g.
\]
By the monotonicity of the perimeter on the admissible class (Definition \ref{def:perimon}) and since $C, \Omega \in \mathcal{A}_C(\M)$ with $C \subsetneq \Omega \subset U$, we have $|\partial \Omega|_g > |\partial C|_g$. Therefore,
\[
\int_C f \dv = k |\partial \Omega|_g > k |\partial C|_g,
\]
which is $(NS)_g$.
\end{proof}

\section{The critical equality case}\label{sec:equality}

We now consider the critical case
\[
\int_C f \dv = k |\partial C|_g.
\]

\begin{proposition}\label{prop:equality}
Assume hypotheses (H1)--(H4) and the $RC$--GNP condition. If
\[
\int_C f \dv = k |\partial C|_g,
\]
then the shape derivative of $J_g$ at $C$ in the direction of the outward normal vanishes:
\[
dJ_g(C)[\nu] = 0.
\]
In particular, $C$ is a critical point of $J_g$ over $\mathcal{A}_C(\M)$.
\end{proposition}

\begin{proof}
The equality $\int_C f \dv = k |\partial C|_g$ together with the flux identity
\[
\int_C f \dv = \int_{\partial C} |\nabla_g u_C|_g \ds
\]
gives
\[
\int_{\partial C} |\nabla_g u_C|_g \ds = k |\partial C|_g.
\]
By Cauchy--Schwarz,
\[
\left(\int_{\partial C} |\nabla_g u_C|_g \ds\right)^2 \leq |\partial C|_g \int_{\partial C} |\nabla_g u_C|_g^2 \ds,
\]
with equality if and only if $|\nabla_g u_C|_g$ is constant on $\partial C$, equal to $k$. Thus
\[
\int_{\partial C} |\nabla_g u_C|_g^2 \ds = k^2 |\partial C|_g.
\]
Substituting into \eqref{eq:shapederiv} with $\Omega = C$ gives
\[
dJ_g(C)[\nu] = \int_{\partial C} \left(k^2 - |\nabla_g u_C|_g^2\right) \ds = 0.
\]
Hence $C$ is a critical point.
\end{proof}

\begin{remark}\label{rem:eqopen}
Proposition \ref{prop:equality} shows that the equality case is critical: the first-order optimality condition is satisfied at $C$. Whether $J_g$ admits a minimizer strictly containing $C$ in this case is a delicate question. Under the additional perimeter monotonicity hypothesis of Theorem \ref{thm:neccond}, no such minimizer can exist: indeed, if $\Omega \supsetneq C$ were a solution, then by the proof of Theorem \ref{thm:neccond} we would have $\int_C f \dv = k|\partial\Omega|_g > k|\partial C|_g$, contradicting the equality. Without this hypothesis, the question remains open and is a natural subject for future work.
\end{remark}

\section{Final remarks}\label{sec:final}

\subsection{Regularity issues}

The proofs above assume hypothesis (H4), namely that the weak solutions $u_\Omega$ are of class $C^{2,\alpha}(\overline{\Omega})$. In general, the existence of regular minimizers is not automatic. Under standard assumptions on $f$ (for instance $f$ bounded away from zero), one can invoke the free boundary regularity theory of Alt--Caffarelli adapted to Riemannian manifolds, as discussed in \cite{GustafssonShahgholian1996}. Alternatively, the results may be interpreted in the weak $H^1$ sense, in which case the free boundary condition is understood in a variational sense.

\subsection{Extension to the $p$-Laplace operator}

The methods developed here can be extended to the $p$-Laplace--Beltrami operator by replacing the Dirichlet energy with the $p$-energy functional
\[
J_p(\omega) = \int_\omega \left(|\nabla_g u_\omega|_g^p - p f u_\omega + k^p\right) \dv,
\]
and using the $\gamma_p$-convergence results on manifolds. Recent work in this direction appears in \cite{DjiteSeck2026b}.

\subsection{Open problems}

The main open problem is to find a sharp geometric condition under which $(NS)_g$ is both necessary and sufficient. The counterexample of Section \ref{sec:counter} shows that nonnegative curvature alone is not enough. Two natural directions are:
\begin{enumerate}
\item Restrict the admissible class to domains contained in a normal convex neighborhood of $C$ and assume perimeter monotonicity, as in Theorem \ref{thm:neccond}.
\item Replace $(NS)_g$ by a different integral condition involving the mean curvature of $\partial C$, which might capture the geometry of the problem more accurately in positive curvature.
\end{enumerate}
We leave these questions for future investigation.

\end{document}